\documentclass[reqno,10pt]{amsart}

\usepackage{amsmath,amssymb,mathrsfs,graphicx,enumitem,verbatim}

\usepackage[pdftex=true,
	bookmarks=true,
	bookmarksopen=true,
	bookmarksnumbered=true,
	pdftoolbar=true,
	pdfmenubar=true,
	pdffitwindow=false,
	pdfstartview={FitH},
	colorlinks=true,
	linkcolor=red,
	citecolor=blue,
	breaklinks=true]{hyperref}

\allowdisplaybreaks

\numberwithin{equation}{section}

\newtheorem{thm}{Theorem}[section]

\newtheorem*{thm*}{Theorem}
\newtheorem*{prop*}{Proposition}
\newtheorem*{cor*}{Corollary}
\newtheorem*{conj*}{Conjecture}

\theoremstyle{definition}

\theoremstyle{remark}
\newtheorem{rmk}[thm]{Remark}

\newcommand{\la}{\langle}
\newcommand{\ra}{\rangle}
\newcommand{\pa}{\partial}

\newcommand{\tn}{\textnormal}

\newcommand{\R}{\mathbb R}
\newcommand{\RR}{\mathbb R}
\newcommand{\C}{\mathbb C}

\newcommand{\N}{\mathbb N}

\newcommand{\Z}{\mathbb Z}

\newcommand{\ep}{\epsilon}
\newcommand{\im}{\operatorname{Im}}
\newcommand{\re}{\operatorname{Re}}

\newcommand{\supp}{\operatorname{supp}}

\newcommand{\wozero}{\setminus o}

\renewcommand{\Im}{\operatorname{Im}}

\newcommand{\Id}{\operatorname{Id}}

\newcommand{\CI}{{\mathcal C}^\infty}
\newcommand{\dCI}{\dot{\mathcal C}^\infty}
\newcommand{\CIc}{{\mathcal C}^\infty_{\mathrm c}}
\newcommand\cO{\mathcal{O}}

\newcommand{\cP}{\mathcal P}

\newcommand{\Vf}{\mathcal V}
\newcommand{\Vb}{\Vf_\bl}

\newcommand{\Psib}{\Psi_\bl}
\newcommand{\bl}{{\mathrm{b}}}

\newcommand{\Hb}{H_{\bl}}

\newcommand{\WFb}{\WF_{\bl}}
\newcommand{\WF}{\mathrm{WF}}
\newcommand{\WFh}{\mathrm{WF}_\semi}
\newcommand{\Op}{\mathrm{Op}}
\newcommand{\Oph}{\mathrm{Op}_h}
\newcommand{\Tb}{{}^{\bl}T}

\newcommand{\Sb}{{}^{\bl}S}

\newcommand{\sH}{\mathsf{H}}

\newcommand{\semi}{h}
\newcommand{\Psih}{\Psi_\semi}

\newcommand{\newconst}{\mathsf{M}}
\newcommand{\normiso}{\mathcal{H}_{\semi,\Gamma}}
\newcommand{\tnormiso}{\mathcal{\tilde H}_{\semi,\Gamma}}
\newcommand{\bnormiso}{\mathcal{H}_{\bl,\Gamma}}
\newcommand{\tbnormiso}{\mathcal{\tilde H}_{\bl,\Gamma}}

\begin{document}
\title[Non-trapping estimates]{Non-trapping estimates near normally hyperbolic trapping}

\author{Peter Hintz and Andras Vasy}
\address{Department of Mathematics, Stanford University, CA
  94305-2125, USA}
\email{phintz@math.stanford.edu}
\email{andras@math.stanford.edu}
\date{November 27, 2013. Final revision: May 12, 2014.}
\thanks{The authors were supported in part by A.V.'s National Science
  Foundation grants DMS-0801226 and DMS-1068742 and P. H.\ was
  supported in part by a Gerhard Casper Stanford Graduate Fellowship
  and the German National Academic Foundation. They are grateful to
  the referee for comments that significantly improved the
  exposition in the manuscript.}

\begin{abstract}
  In this paper we prove semiclassical resolvent estimates for operators with normally hyperbolic trapping which are lossless relative to non-trapping estimates but take place in weaker function spaces. In particular, we obtain non-trapping estimates in standard $L^2$ spaces for the resolvent sandwiched between operators which localize away from the trapped set $\Gamma$ in a rather weak sense, namely whose principal symbols vanish on $\Gamma$.
\end{abstract}

\maketitle

\section{Introduction}
The purpose of this paper is to obtain semiclassical estimates for
pseudodiffererential operators $P_h(z)$ with normally hyperbolic
trapping for $z$ real which are lossless relative to non-trapping
estimates, but take place in weaker function spaces which are defined
in a manner related to the Hamiltonian dynamics. Thus, the main result is an estimate of the form
\[
  \|u\|_{\normiso}\leq Ch^{-1}\|P_h(z)u\|_{\normiso^*},
\]
with certain function spaces $\normiso$ and $\normiso^*$, described
below; away from the trapped set these are just standard $L^2$ spaces.
As the main application of such estimates is in so-called
b-spaces, e.g.\ Kerr-de Sitter spaces, for which the estimates follow from the semiclassical ones
immediately in the presence of dilation invariance, we also prove
their counterpart in the general, non-dilation-invariant, b-setting. Extensions of special cases of these estimates play an important role in the recent global analysis of nonlinear wave equations on asymptotically Kerr-de Sitter spaces by the authors \cite{HintzVasyQuasilinearKdS}.

So at first we consider a family $P_h(z)$ of semiclassical pseudodifferential operators $P_h(z)\in\Psih(X)$ on a closed manifold $X$, depending smoothly on the parameter $z\in\C$, with normally hyperbolic trapping at the trapped set $\Gamma$, and assume that $P_h(z)$ is formally self-adjoint near $\Gamma$ for $z\in\R$; moreover, we add complex absorption $W$ in such a way that all forward and backward bicharacteristics outside $\Gamma$ either enter the elliptic set of $W$ in finite time or tend to $\Gamma$, and in at least one of the two directions they tend to the elliptic set of $W$. The bicharacteristics tending to $\Gamma$ in the forward/backward directions are forward/backward trapped; denote by $\Gamma_-$, resp. $\Gamma_+$ the forward, resp.\ backward trapped set,\footnote{In the notation of Wunsch and Zworski \cite{Wu11}, which we recall below in Section~\ref{SubsecSemiNotation}, $\Gamma_\pm$ are the backward/forward trapped sets for all (not necessarily null) bicharacteristics near the, say, zero level set of the semiclassical principal symbol $p_{\semi,z}$, and $\Gamma_\pm^\lambda$ are the corresponding sets within the $\lambda$-level set of $p_{\semi,z}$.} and assume that these are smooth codimension one submanifolds of $T^*X$ which intersect transversally in $\Gamma$, which we moreover assume to be symplectic.

In this normally hyperbolic setting, under additional hypotheses, Wunsch and Zworski \cite{Wu11} have shown polynomial semiclassical resolvent estimates
\begin{equation}\label{eq:global-estimate}
\|u\|\leq Ch^{-N}\|P_h(z) u\|,\ 0<h<h_0,
\end{equation}
in small strips $|\im z|\leq ch$, $c>0$ sufficiently small, $N>1$, and indeed for $z$ real, the loss (as compared to non-trapping estimates, which hold in many cases where there is no trapping, and which lose a power of $h^{-1}$) is merely logarithmic, i.e.\ one has
\begin{equation}\label{eq:log-global-estimate}
\|u\|\leq Ch^{-1}(\log h^{-1})\|P_h(z) u\|,\ 0<h<h_0,
\end{equation}
where $\|\cdot\|$ is the $L^2$-norm; Bony, Burq and Ramond \cite{BonyBurqRamondTrapping} showed that \eqref{eq:log-global-estimate} is indeed sharp. Dyatlov \cite{DyatlovSpectralGaps} improved these estimates in $\im z<0$ by making $c$ and $N$ explicit; in a more general setting, Nonnenmacher and Zworski \cite{NonnenmacherZworskiDecay} obtained the optimal value for $c$. 

We are concerned with improved estimates (for $z$ almost real) if one localizes $u$ and $P_h(z)u$ away from the trapping $\Gamma$ in a rather weak sense, such as by applying pseudodifferential operators with symbols vanishing at $\Gamma$. To place this into context, recall that Datchev and Vasy \cite{Da12, Datchev-Vasy:Semiclassical} have shown that under our assumptions, with $\im z=\cO(h^\infty)$, if $A,B\in\Psih(X)$ with $\WFh'(A)\cap\Gamma=\WFh'(B)\cap\Gamma=\emptyset$, $B$ elliptic on $\WFh'(A)$, then for all $M$ there is $N$ such that
\begin{equation}\label{eq:Datchev-Vasy-est}
\|Au\|\leq Ch^{-1}\|BP_h(z)u\|+C'h^M\|u\|+C''h^{-N}\|(\Id-B)P_h(z)u\|.
\end{equation}
Thus, if $P_h(z)u$ is $\cO(h^{N-1})$ at $\Gamma$ (corresponding to the $\Id-B$ term in the estimate), then on the elliptic set of $A$, hence off $\Gamma$ by appropriate choice of $A$, $u$ satisfies non-trapping semiclassical estimates:
$$
\|AP_h(z)^{-1}A v\|\leq Ch^{-1}\|v\|,
$$
with $A$ as above (take $B$ as above with $\WFh'(\Id-B)\cap\WFh'(A)=\emptyset$).
Here the $\cO(h^\infty)$ bound on $\im z$ arises from the a priori estimate,
\eqref{eq:global-estimate}, and if $1<N<3/2$, e.g.\ as is on, and
sufficiently near, the real axis,\footnote{In the latter case by
the Phragm\'en-Lindel\"of theorem.} then one can take $\im z=\cO(h^{-1+2N})$. The
purpose of this paper is to improve this result by relaxing the
conditions on $\WFh'(A)$ and $\WFh'(B)$ in
\eqref{eq:Datchev-Vasy-est}.

The main point of the theorem below is thus that its estimate degenerates only at, as opposed to near, $\Gamma$. The proof given here is closely related to the proof of Wunsch and Zworski \cite[\S4]{Wu11} but can take place in a significantly simpler, standard semiclassical pseudodifferential algebra, at the cost of being suboptimal in terms of the $L^2$-estimate, even though it is optimal (i.e.\ non-trapping) when a pseudodifferential operator with vanishing principal symbol at $\Gamma$ is applied from both sides. To set this up, let $Q_\pm\in\Psih^{-\infty}(X)$ be self-adjoint and have symbols which are defining functions of $\Gamma_\pm$ near $\Gamma$, say on a neighborhood $O$ of $\Gamma$. Let $Q_0\in\Psih^0(X)$ be a semiclassical operator with $\WFh'(Q_0)\cap\Gamma=\emptyset$ which is elliptic on $O^c$ (and thus on a neighborhood of $O^c$), with real principal symbol for convenience. One considers {\em normally isotropic} spaces at $\Gamma$, denoted $\normiso$, with squared norms given by
\[
  \|u\|_{\normiso}^2=\|Q_0 u\|^2+\|Q_+ u\|^2+\|Q_- u\|^2+h\|u\|^2;
\]
this is just the standard $L^2$-space microlocally away from $\Gamma$ as one of $Q_+$, $Q_-$ or $Q_0$ is elliptic there, and it does not
depend on the choice of $Q_0$ as on $O\setminus\Gamma$ one of $Q_+$ and $Q_-$ is elliptic at every point. The dual space relative to $L^2$ is then\footnote{One really has $Q_\pm^*$ and $Q_0^*$ in this formula, but the reality of the principal symbols assures that one may replace them by $Q_\pm$ and $Q_0$ modulo $hL^2$. See \cite[Appendix~A]{Melrose-Vasy-Wunsch:Corners} for a general discussion of the underlying functional analysis; also see Footnote~\ref{FootDual}.}
$$
\normiso^*=h^{1/2}L^2+Q_+L^2+Q_-L^2+Q_0L^2
$$
(which is $L^2$ as a space, but with this norm); $P_h(z)u$ will then be measured in $\normiso^*$.

\begin{thm}\label{thm:Gamma}
Let $P=P_h(z)$, $Q_\pm$ be as above, $\Im z=\cO(h^2)$. Then
\begin{equation}\label{eq:weaker-nontrapping-estimate}
\|Q_+ u\|+\|Q_- u\|\leq Ch^{-1}\|Pu\|_{\normiso^*}+C'h^{1/2}\|u\|,
\end{equation}
and thus by \eqref{eq:log-global-estimate},
\begin{equation}\label{eq:stronger-nontrapping-estimate}
\|u\|_{\normiso}\leq Ch^{-1}\|Pu\|_{\normiso^*}.
\end{equation}
\end{thm}

In fact, we also obtain a direct proof of \eqref{eq:stronger-nontrapping-estimate} without using \eqref{eq:log-global-estimate} at the end of Section~\ref{SecSemiclassical}. Note that this theorem in particular implies the main result of \cite{Da12} in this setting, in that the estimates are of the same kind, except that in \cite{Da12} $Pu$ is assumed to be microlocalized away from $\Gamma$, and $u$ is estimated microlocally away from $\Gamma$.

The aforementioned b-estimates will be proved in Section~\ref{sec:b}, see Theorem~\ref{thm:b-normiso-propagation}.

\section{Semiclassical resolvent estimates on the real line}
\label{SecSemiclassical}

\subsection{Notation and definitions}
\label{SubsecSemiNotation}

We will review some definitions of semiclassical analysis, partially in order to fix our notation. For a general reference, see Zworski \cite{ZworskiSemiclassical}.

Let $X$ be a compact $n$-dimensional manifold without boundary, and fix a smooth density on $X$.

\begin{itemize}[leftmargin=0.8cm]
\item For $u\in L^2(X)$, denote by $\|u\|$ its $L^2(X)$ norm; moreover, denote by $\la\cdot,\cdot\ra$ the (sesquilinear) inner product on $L^2(X)$.
\item A family of functions $u=(u_h)_{h\in(0,1)}$ on $X$ is \emph{polynomially bounded} if $\|u\|\leq Ch^{-N}$ for some $N$. If $k\in\R$, we say that $u\in\cO(h^k)$ if $\|u\|\leq C_k h^k$, and $u\in\cO(h^\infty)$ if $\|u\|\leq C_N h^N$ for every $N$.
\item For $a=(a_h)_{h\in(0,1)}\in\CI(T^*X)$, we say $a\in h^k S^m(T^*X)$ if $a$ satisfies
  \[
    |\pa_z^\alpha\pa_\zeta^\beta a_h(z,\zeta)|\leq C_{\alpha\beta}h^k\la\zeta\ra^{m-|\beta|}
  \]
  for all multiindices $\alpha,\beta$ and all $N\in\N$ in any coordinate chart, where the $z$ are coordinates in the base and $\zeta$ coordinates in the fiber. We define the \emph{semiclassical quantization} $\Oph(a)$ of $a$ by
  \[
    \Oph(a)u(z)=(2\pi h)^{-n}\int e^{iz\zeta/h}a(z,\zeta)\hat u(\zeta/h)\,d\zeta
  \]
  for $u\in\CIc(X)$ supported in a chart and for general $u\in\CIc(X)$
  by using a partition of unity. We write $\Oph(a)\in
  h^k\Psih^m(X)$. The quantization depends on the choice of partition
  of unity, but the resulting class of operators does not, modulo
  operators that have Schwartz kernel in $h^\infty\CI(X^2)$. We say
  that $a$ is a \emph{symbol} of $\Oph(a)$. The equivalence class of
  $a$ in $h^k S^m(T^*X)/h^{k-1} S^{m-1}(T^*X)$ is invariantly defined
  and is called the \emph{principal symbol} of $\Oph(a)$. All operators below except $Q_0\in\Psih^0(X)$ will in fact have compact microsupport in the sense that they are quantizations of symbols $a\in h^k S^m(T^*X)$ satisfying in addition for all $N$
  \[
    |\pa_z^\alpha\pa_\zeta^\beta a_h(z,\zeta)|\leq C_N h^N\la\zeta\ra^{-N}\tn{ for all multiindices }\alpha,\beta
  \]
  for $\zeta$ outside of a compact subset of $T^*X$. We denote the class of such symbols by $h^k S(T^*X)$ and the corresponding class of operators by $h^k\Psih(X)$.
\item If $A,B\in\Psih(X)$, then $[A,B]\in h\Psih(X)$, and its principal symbol is $\frac{h}{i}\sH_a b$, where we define the \emph{Hamilton vector field} in a coordinate chart by
  \[
    \sH_a = (\pa_\zeta a)\pa_z - (\pa_z a)\pa_\zeta.
  \]
\item By a \emph{bicharacteristic} of $A$ we mean an integral curve of the Hamilton vector field of the principal symbol of $A$. We denote the integral curve passing through the point $\rho\in T^*X$ by $\gamma_\rho$, i.e.\ $\gamma_\rho(0)=\rho$ and $\gamma_\rho'(s)=\sH_a(\gamma_\rho(s))$. We shall also write $\phi^s(\rho):=\gamma_\rho(s)$ for the bicharacteristic flow.
\item For a polynomially bounded family $(u_h)_{h\in(0,1)}$ and $k\in\R\cup\{\infty\}$, we say that $u=\cO(h^k)$ at a point $\rho\in T^*X$ if there exists $a\in S(T^*X)$ with $a(\rho)\neq 0$ such that $\|\Oph(a)u\|=\cO(h^k)$. We define the \emph{semiclassical wave front set} $\WFh(u)$ of $u$ as the complement of the set of all $\rho\in T^*X$ at which $u=\cO(h^\infty)$.
\item The \emph{microsupport} of $A=\Oph(a)\in h^k\Psih(X)$, denoted $\WFh'(A)$, is the complement of the set of all $\rho\in T^*X$ so that $|\pa^\alpha a|=\cO(h^\infty)$ near $\rho$ for every multiindex $\alpha$, in any (and therefore in every) coordinate chart.
\item For $A\in h^k\Psih(X)$ with principal symbol $a\in h^kS(T^*X)$, we say that $A$ is \emph{elliptic} at $\rho\in T^*X$ if there is a constant $C>0$ such that $|a(\rho')|\geq Ch^k$ for $\rho'$ near $\rho$ and $h$ sufficiently small. For a subset $E\Subset T^*X$, we say that $A$ is elliptic on $E$ if $A$ is elliptic at each point of $E$. If $A\in h^k\Psih(X)$ is elliptic on $E\Subset T^*X$ and $Au=f$ with $u,f$ polynomially bounded and $f$ is $\cO(1)$ on $E$, then \emph{microlocal elliptic regularity} states that $u$ is $\cO(h^{-k})$ on $E$.
\item The \emph{semiclassical characteristic set} of the semiclassical operator $A\in\Psih(X)$ with principal symbol $a$ is defined by $\Sigma_\semi=\{\rho\in T^*X\colon a(\rho)=0\}$.
\item If $A\in\Psih(X)$ has a principal symbol with non-positive imaginary part, $u,f$ are polynomially bounded, $Au=f$, and $u=\cO(h^k)$ at $\rho$, $f=\cO(h^{k+1})$ on $\gamma_\rho([0,T])$ for some $T>0$, then the \emph{propagation of singularities} states that $u=\cO(h^k)$ at $\gamma_\rho(T)$.
\item Let $P\in\Psih(X)$ be a semiclassical operator. Let $U\subset X$ denote an open subset so that the cotangent bundle over $U$ contains what will be the trapped set, and place complex absorbing potentials in a neighborhood of $U^c$.\footnote{See \cite{Wu11,Va12} for details; the point here is that the relevant part of our analysis takes place microlocally near the trapped set, and the complex absorbing potentials allow us to `cut off' the bicharacteristic flow in a neighborhood of the trapped set.} We recall the notion of \emph{normal hyperbolicity} from \cite{Wu11}: Define the \emph{backward}, resp.\ \emph{forward, trapped set} $\Gamma_+$, resp.\ $\Gamma_-$, by
  \[
    \Gamma_\pm=\{\rho\in T^*X\colon \gamma_\rho(s)\notin T^*_{U^c} X\tn{ for all }\mp s\geq 0\}.
  \]
  Let $\Gamma_\pm^\lambda=\Gamma_\pm\cap p^{-1}(\lambda)$ be the backward/forward trapped set within the energy surface $p^{-1}(\lambda)$, and define the \emph{trapped set} $\Gamma_\lambda:=\Gamma_+^\lambda\cap\Gamma_-^\lambda$. We say that $P$ is \emph{normally hyperbolically trapping} if:
  \begin{enumerate}
    \item There exists $\delta>0$ such that $dp\neq 0$ on $p^{-1}(\lambda)$ for $|\lambda|<\delta$;
	\item $\Gamma_\pm\cap p^{-1}(-\delta,\delta)$ are smooth codimension one submanifolds intersecting transversally at $\Gamma\cap p^{-1}(-\delta,\delta)$, and $\Gamma\cap p^{-1}(-\delta,\delta)$ is symplectic;
	\item the flow is hyperbolic in the normal directions to $\Gamma_\lambda$ within the energy surface: There exist subbundles $E^\pm_\lambda$ of $T_{\Gamma_\lambda}(\Gamma_\pm^\lambda)$ such that $T_{\Gamma_\lambda}\Gamma_\pm^\lambda=T\Gamma_\lambda\oplus E^\pm_\lambda$, where $d\phi^s\colon E^\pm_\lambda\to E^\pm_\lambda$, and there exists $\theta>0$ such that for all $|\lambda|<\delta$
	\[
	  \|d\phi^s(v)\|\leq Ce^{-\theta|t|}\|v\|\tn{ for all }v\in E^\mp_\lambda, \pm t\geq 0.
	\]
  \end{enumerate}
\end{itemize}

\subsection{Details on the setup and proof of the main result}
\label{SubsecSemiResult}

Let $p=p_{\semi,z}$ be the semiclassical principal symbol of $P=P_h(z)$. Recall from the work of Wunsch and Zworski \cite[Lemma~4.1]{Wu11}, with a corrected argument in \cite{Wu11c}, that for defining functions $\phi_\pm$ of $\Gamma_\pm$ (near $\Gamma$, namely in a neighborhood $O$ of $\Gamma$) one can take $\phi_\pm$ with
\[
\sH_p\phi_\pm=\mp c_\pm^2\phi_\pm
\]
with $c_\pm>0$ near $\Gamma$, and with\footnote{These defining functions exist globally when $\Gamma_\pm$ is orientable; but even if $\Gamma_\pm$ is not such, the square is globally defined. There is only a minor change required below if $\phi_\pm$ are not well defined; see Footnote~\ref{footnote:non-orientable}.}
\[
\{\phi_+,\phi_-\}>0
\]
near $\Gamma$. {\em This is the only relevant feature of normal hyperbolicity for this paper; thus these identities and estimates could be taken as its definition for our purposes.} By shrinking $O$ if necessary we may assume that this Poisson bracket as well as $c_\pm$ have positive lower bounds on $O$. Then notice that
\[
\sH_p\phi_+^2=-2 c_+^2\phi_+^2,\qquad \sH_p\phi_-^2=2c_-^2\phi_-^2.
\]
As indicated in the introduction, we consider {\em normally isotropic} spaces at $\Gamma$, denoted $\normiso$, with squared norms given by
\[
  \|u\|_{\normiso}^2=\|Q_0 u\|^2+\|Q_+ u\|^2+\|Q_- u\|^2+h\|u\|^2;
\]
we can take $Q_\pm$ with principal symbol $\phi_\pm$, while $Q_0$ is elliptic on $O^c$ with real principal symbol. This is just the standard $L^2$-space microlocally away from $\Gamma$ as one of $Q_+$, $Q_-$ and $Q_0$ is elliptic there, and it does not depend on the choice of $Q_0$ as on $O\setminus\Gamma$ one of $Q_+$ and $Q_-$ is elliptic at every point. Notice that in fact
\[
(Q_+-iQ_-)^*(Q_+-iQ_-)=Q_+^*Q_+ + Q_-^*Q_- - i[Q_+,Q_-]
\]
and if $B\in\Psih(X)$ with $\WFh'(B)\subset O$ then
$$
h\|Bv\|^2\leq C\re\la i[Q_+,Q_-]Bv,Bv\ra+Ch^{N'}\|v\|^2,
$$
$C>0$, in view of $\{\phi_+,\phi_-\}>0$ on $O$, so
$$
Q_+^*Q_++Q_-^*Q_-=\frac{1}{2}(Q_+^*Q_++Q_-^*Q_-+(Q_+-iQ_-)^*(Q_+-iQ_-)+i[Q_+,Q_-])
$$
shows that, for $h>0$ small, the norm on $\normiso$ is equivalent to just the norm
$$
\|u\|^2_{\normiso,2}=\|Q_0 u\|^2+\|Q_+ u\|^2+\|Q_- u\|^2.
$$
As mentioned in the introduction, the dual space relative to $L^2$ is then
$$
\normiso^*=h^{1/2}L^2+Q_+L^2+Q_-L^2+Q_0L^2.
$$
Then $\Psih(X)$ acts on $\normiso$, and thus on $\normiso^*$, for
$B\in\Psih(X)$ preserves $h^{-1/2}L^2$ and gives
$$
\|Q_+Bu\|\leq \|BQ_+u\|+\|[Q_+,B]u\|\leq C\|Q_+u\|+h\|u\|_{L^2},
$$
with a similar result for $Q_-$ and $Q_0$. We remark that the notation $\normiso$ is justified as the space depends only on $\Gamma$, not on the particular defining functions $\phi_\pm$ as any other defining functions would change $Q_\pm$ by an elliptic factor modulo an element of $h\Psih(X)$, whose contribution to the squared norm can be absorbed into $Ch^2\|u\|^2_{L^2}$, and thus dropped altogether (for $h$ small) in view of the
equivalence of the two norms discussed above.

We are now ready to prove Theorem~\ref{thm:Gamma}. We remark that the microlocal version of the two estimates of the theorem is that given any neighborhood $O'$ of $\Gamma$ with closure in $O$, there exist $B_0\in\Psih(X)$ elliptic at $\Gamma$, $B_1,B_2\in\Psih(X)$ with $\WFh'(B_2)\cap\Gamma_+=\emptyset$, $\WFh'(B_j)\subset O'$ for $j=0,1,2$ such that
\begin{equation}\label{eq:mweaker-nontrapping-estimate}
\|B_0Q_+ u\|+\|B_0Q_- u\|\leq h^{-1}\|B_1Pu\|_{\normiso^*}+\|B_2u\|_{L^2}+C'h^{1/2}\|u\|_{L^2},
\end{equation}
respectively
\begin{equation}\label{eq:mstronger-nontrapping-estimate}
\|B_0u\|_{\normiso}\leq h^{-1}\|B_1Pu\|_{\normiso^*}+\|B_2 u\|_{L^2}+C'h\|u\|_{L^2};
\end{equation}
see \eqref{eq:microloc-iso-est}.
The theorem is then proved by controlling the $B_2 u$ term using the
backward non-trapped nature of $\Gamma_-\setminus\Gamma$.

\begin{proof}[Proof of Theorem~\ref{thm:Gamma}]
We first prove \eqref{eq:weaker-nontrapping-estimate}, which proves
\eqref{eq:stronger-nontrapping-estimate} by
\eqref{eq:log-global-estimate}. In fact, one can also give a direct proof of \eqref{eq:stronger-nontrapping-estimate} without using \eqref{eq:log-global-estimate}; see the discussion following this proof.

Let $\chi_0(t)=e^{-1/t}$ for $t>0$,
$\chi_0(t)=0$ for $t\leq 0$, $\chi\in\CI_c([0,\infty))$ be identically $1$ near $0$ with $\chi'\leq
0$, and indeed with $\chi'\chi=-\chi_1^2$, $\chi_1\geq 0$,
$\chi_1\in\CI_c([0,\infty))$,
and let $\psi\in\CI_c(\RR)$ be identically $1$ near $0$.
Let
$$
a=\chi_0(\phi_+^2-\phi_-^2+\kappa)\chi(\phi_+^2)\psi(p),
$$
$\kappa>0$ small.
Notice that on $\supp a$, if $\chi$ is supported in $[0,R]$,
$$
\phi_+^2\leq R,\ \phi_-^2\leq\phi_+^2+\kappa=R+\kappa,
$$
so $a$ is localized near $\Gamma$ if $R$ and $\kappa$ are taken
sufficiently small.
Then
\begin{equation*}\begin{split}
\frac{1}{4}\sH_p(a^2)=&-(c_+^2\phi_+^2+c_-^2\phi_-^2)(\chi_0\chi_0')(\phi_+^2-\phi_-^2+\kappa)
\chi(\phi_+^2)^2\psi(p)^2\\
&- c_+^2\phi_+^2(\chi'\chi)(\phi_+^2)\chi_0(\phi_+^2-\phi_-^2+\kappa)^2\psi(p)^2.
\end{split}\end{equation*}
Now $\chi_0'\geq 0$, so the two terms have opposite
signs. Let\footnote{If $\phi_\pm$ is not defined globally, $a_\pm$ are\label{footnote:non-orientable}
not defined as stated. (The term $e_-^2$ need not have a sign, so this
issue does not arise for it; see the Weyl quantization argument
below.) However, $a_\pm$ need not be real below, so as long as one can
choose $\psi_\pm$ complex valued with $|\psi_\pm|^2=\phi_\pm^2$,
replacing the first factor of $\phi_\pm$ with $\psi_\pm$ in the
definition of $a_\pm$ allows one to complete the argument in general.}
$$ 
a_\pm=\phi_\pm \sqrt{(\chi_0\chi_0') (\phi_+^2-\phi_-^2+\kappa)}\chi(\phi_+^2)\psi(p),
$$ 
and
$$
e_-=c_+\phi_+\chi_1(\phi_+^2) \chi_0(\phi_+^2-\phi_-^2+\kappa)\psi(p);
$$
then
\begin{equation}
\label{eq:Gamma-commutator}
\frac{1}{4}\sH_p(a^2)=-c_+^2a_+^2-c_-^2a_-^2+e_-^2.
\end{equation}
Here
\begin{equation*}\begin{split}
&\supp e_-\subset\supp a,\\
&\supp e_-\cap\Gamma_+=\emptyset,
\end{split}\end{equation*}
with the last statement following from $\phi_+^2$ taking values away
from $0$ on $\supp\chi_1$; see Figure~\ref{FigTrapping}.

\begin{figure}[!ht]
  \centering
  \includegraphics{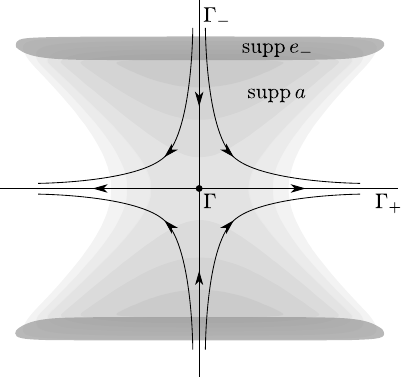}
  \caption{Supports of the commutant $a$ and the error term $e_-$ in the positive commutator argument of the non-trapping estimate near the trapped set $\Gamma$, Theorem~\ref{thm:Gamma}. The support of $a$ is indicated in light gray; on $\supp a\setminus\supp e_-$, darker colors correspond to larger values of $a$. Also shown are the forward, resp.\ backward, trapped set $\Gamma_-$, resp. $\Gamma_+$, and the bicharacteristic flow nearby. The figure already suggests that $\sH_p(a^2)$ is non-positive away from $\supp e_-$, and actually negative away from $\supp e_-\cup\Gamma$; see equation~\eqref{eq:Gamma-commutator}.}
  \label{FigTrapping}
\end{figure}

One then takes $A\in\Psih(X)$ with principal symbol $a$, and with
$\WFh'(A)\subset\supp a$,
$A_\pm\in\Psih(X)$ with principal symbols of $a_\pm$, and with
$\WFh'(A_\pm)\subset\supp a_\pm$, $C_\pm$ have symbol $c_\pm$ and with
$\WFh'(C_\pm)\subset\supp c_\pm$; one
similarly lets $E_-\in\Psih(X)$ have principal symbol $e_-$,
and wave front set in the support of $e_-$.
This gives that
\begin{equation}\label{eq:pos-commutator-trap-weak}
\frac{i}{4 h}[P,A^*A]=-(C_+A_+)^*(C_+A_+)-(C_-A_-)^*(C_-A_-)+E_-^*E_-+hF,
\end{equation} 
for some $F\in\Psih(X)$ with
$$
\WFh'(F)\subset\supp a.
$$ 
Thus
$$
\frac{i}{4h}\la[P,A^*A]u,u\ra=-\|C_+A_+ u\|^2 -\|C_-A_- u\|^2+\|E_- u\|^2+h\la Fu,u\ra.
$$
Expanding the left hand side gives
\begin{equation*}\begin{split}
&\la PA^*Au,u\ra-\la A^*AP u,u\ra\\
&\qquad=\la Au,A Pu\ra-\la AP u,Au\ra+\la (P-P^*)A^*Au,u\ra.
\end{split}\end{equation*}
As we are assuming that $P-P^*$ is $\cO(h^2)$ near $\Gamma$, we may also assume that this holds on $\supp a$, thus the last term is $\cO(h^2)\|u\|^2$. Thus,
\begin{equation}\label{eq:pos-comm-est-trap}
\|C_+A_+ u\|^2+\|C_-A_- u\|^2\leq \|E_- u\|^2+h^{-1}|\la 
APu,Au\ra|+C_1 h\|u\|^2.
\end{equation} 
Now, by the duality of $\normiso$ and $\normiso^*$ relative to the $L^2$ inner product,
\begin{equation}
\label{EqPairing}
|\la APu,Au\ra|\leq \|APu\|_{\normiso^*}\|Au\|_{\normiso}\leq \frac{h\ep}{2}\|Au\|^2_{\normiso}+\frac{1}{2h\ep}\|APu\|_{\normiso^*}^2.
\end{equation}
Further, for $\ep>0$ small, $\ep\|Q_+Au\|^2$ can be estimated in terms
of $\|C_+A_+u\|^2+\cO(h)\|u\|^2$, as can be seen by comparing the
principal symbols, in particular using the ellipticity of $C_+$ on
$\supp a$. One can thus absorb $\frac{\ep}{2}\|Au\|^2_{\normiso}$
into the left hand side of \eqref{eq:pos-comm-est-trap}.
This shows
$$
\|C_+A_+ u\|^2+\|C_-A_- u\|^2\leq C\|E_- u\|^2+Ch^{-2}\|APu\|_{\normiso^*}^2+Ch\|u\|^2.
$$
Now, since the region $\supp e_-$ is disjoint from $\Gamma_+$, it is backward non-trapped, and thus the standard propagation of singularities with complex absorption (see e.g.\ \cite[Lemma~5.1]{Datchev-Vasy:Semiclassical}) implies that $E_-u$ is controlled by $Pu$ microlocalized off $\Gamma_+$, hence by $Q_+Pu$, modulo higher order (in $h$) terms in $Pu$. This proves the first part of Theorem~\ref{thm:Gamma} since $A_\pm$ is an elliptic multiple of $Q_\pm$ microlocally near $\Gamma$. Thus, if we have a bound $\|u\|\leq C'h^{-1-s}\|Pu\|_{L^2}$, $0<s<1/2$, and thus $h\|u\|^2\leq C'h^{-1-2s}\|Pu\|_{L^2}^2\leq C''h^{-1-2s}\|Pu\|_{\normiso^*}^2$, this implies a non-trapping estimate:
$$
\|u\|_{\normiso}\leq Ch^{-1}\|Pu\|_{\normiso^*}.
$$
This completes the proof of Theorem~\ref{thm:Gamma}.
\end{proof}

In fact, as mentioned earlier, a slight change of point of view proves
Theorem~\ref{thm:Gamma} directly. To see this, we use the Weyl
quantization\footnote{The Weyl quantization is actually irrelevant. It\label{footnote:Weyl}
  is straightforward to see that if $A\in\Psih(X)$ and if the
  principal symbol of $A$ is real then the real part
  of the subprincipal symbol is defined independently of choices. This
  is all that is needed for the argument below.}
when choosing $a,a_\pm,c_\pm,e_-$; since we are on a
manifold, this requires identifying functions with half-densities via
trivialization of the half-density bundle by the Riemannian
metric; this identification preserves self-adjointness. We also write
$P_{\semi,z}$ as the Weyl quantization of $p_0+hp_1$ with $p_0$, $p_1$
real modulo $\cO(h^2)$. Then the principal symbol calculation above
holds with $p_0$ in place of $p$, and with $p_1$ included it yields
additional terms
\begin{equation*}\begin{split}
\frac{1}{4}\sH_p(a^2)=&-(c_+^2\phi_+^2+c_-^2\phi_-^2-h\phi_+\sH_{p_1}\phi_++h\phi_-\sH_{p_1}\phi_-)\\
&\qquad\qquad\qquad\qquad\times(\chi_0\chi_0')(\phi_+^2-\phi_-^2+\kappa)
\chi(\phi_+^2)^2\psi(p)^2\\
&- (c_+^2\phi_+^2-h\phi_+\sH_{p_1}\phi_+)(\chi'\chi)(\phi_+^2)\chi_0(\phi_+^2-\phi_-^2+\kappa)^2\psi(p)^2.
\end{split}\end{equation*}
Now, \eqref{eq:pos-commutator-trap-weak} becomes
\begin{equation}\begin{split}\label{eq:pos-commutator-trap-weak-Weyl}
\frac{i}{4 h}[P,A^*A]=&-(C_+A_+)^*(C_+A_+)-(C_-A_-)^*(C_-A_-)\\
&\qquad+h(A_+^*G_++G_+^*A_++A_-^*G_-+G_-^*A_-)+E+h^2F,
\end{split}\end{equation}
with $G_\pm$ being the Weyl quantization of
$$
g_\pm=\pm\frac{1}{2}(\sH_{p_1}\phi_\pm) \sqrt{(\chi_0\chi_0') (\phi_+^2-\phi_-^2+\kappa)}\chi(\phi_+^2)\psi(p),
$$
and with $F\in\Psih(X)$ with
$$
\WFh'(F)\subset\supp a.
$$
Correspondingly, \eqref{eq:pos-comm-est-trap} becomes
\begin{equation}\begin{split}\label{eq:pos-comm-est-trap-Weyl}
\|C_+A_+ u\|^2+\|C_-A_- u\|^2\leq&\ |\la Eu,u\ra|+h^{-1}|\la
APu,Au\ra|\\
&\quad+2h\|A_+ u\|\|G_+u\|+2h\|A_-u\|\|G_- u\|+C_1 h^2\|u\|^2.
\end{split}\end{equation}
The terms with $G_\pm$ on the right hand side can be estimated by
$$
\ep\|A_+ u\|^2+\ep^{-1}h^2\|G_+u\|^2+\ep\|A_- u\|^2+\ep^{-1}h^2\|G_-u\|^2,
$$
and for $\ep>0$ sufficiently small, the $\|A_\pm u\|^2$ terms can now
be absorbed into the left hand side of
\eqref{eq:pos-comm-est-trap-Weyl}. Proceeding as above yields
\begin{equation}\label{eq:microloc-iso-est}
\|C_+A_+ u\|^2+\|C_-A_- u\|^2\leq C|\la
Eu,u\ra|+Ch^{-2}\|APu\|_{\normiso^*}^2+Ch^2\|u\|^2.
\end{equation}
Together with the non-trapping for the $E$ term this gives the global estimate
$$
\|u\|_{\normiso}^2\leq Ch^{-2}\|Pu\|_{\normiso^*}^2+Ch^2\|u\|^2,
$$
and now the last term on the right hand side can be absorbed into the
left hand side for sufficiently small $h$, giving the estimate \eqref{eq:stronger-nontrapping-estimate}.

Notice that this also directly gives a weaker version of the Wunsch-Zworski estimate \eqref{eq:log-global-estimate}, namely
$$
\|u\|_{L^2}\leq Ch^{-2}\|Pu\|_{L^2},
$$
in view of the continuity of the inclusions $\normiso\hookrightarrow h^{-1/2}L^2$ and $h^{1/2}L^2\hookrightarrow\normiso^*$.

\begin{rmk}
\label{RmkEll}
  If one is interested in a fixed operator, rather than in a parameter-dependent family of operators, one can naturally strengthen the estimates \eqref{eq:weaker-nontrapping-estimate}--\eqref{eq:stronger-nontrapping-estimate} by adding $h^{-1}\|\hat P_0 u\|$ to the left hand sides, where $\hat P_0$ is any elliptic multiple of $P$. A more natural way of phrasing such an improvement is to use `coisotropic, normally isotropic' spaces $\tnormiso$ and $\tnormiso^*$, where the squared norm on $\tnormiso$ is defined by
  \[
    \|u\|_{\tnormiso}^2=\|Q_0 u\|^2+\|Q_+ u\|^2+\|Q_- u\|^2+h^{-1}\|\hat P_0 u\|^2+h\|u\|^2,
  \]
  which strengthens the space and therefore weakens its dual. Using these spaces instead in \eqref{EqPairing}, one obtains an additional term from $h\|Au\|_{\tnormiso}^2$, namely $\|\hat P_0 Au\|^2$, which is bounded by $C(h^{-1}\|AP u\|_{\tnormiso^*}^2+h\|u\|^2)$, and thus the remainder of the second proof goes through.
\end{rmk}

\section{Non-trapping estimates in non-dilation invariant settings}\label{sec:b}

We now transfer Theorem~\ref{thm:Gamma} into the b-setting; the discussion in the previous section is essentially the dilation invariant special case of this, as we will explain below, though in the b-setting there is additional localization near the boundary. One main application of the b-estimate is in the analysis of linear and non-linear waves on asymptotically Kerr-de Sitter spaces; see \cite{Va12,HintzVasyQuasilinearKdS} for details.

\subsection{Notation and definitions}
\label{SubsecBNotation}

For a general reference for b-analysis, see Melrose \cite{Me93}.

Let $M$ be an $n$-dimensional compact manifold with boundary $X$.

\begin{itemize}[leftmargin=0.8cm]
\item Let $\Vb(M)$ be the Lie algebra of \emph{b-vector fields} on $M$, i.e.\ of vector fields on $M$ which are tangent to $X$. Elements of $\Vb(M)$ are sections of a natural vector bundle on $M$, namely the \emph{b-tangent bundle} $\Tb M$; in local coordinates $(\tau,x)$ near $X$, the fibers of $\Tb M$ are spanned by $\tau\pa_\tau$ and $\pa_x$. The fibers of the dual bundle $\Tb^* M$, called \emph{b-cotangent bundle}, are spanned by $\frac{d\tau}{\tau}$ and $dx$.

It is often convenient to consider the fiber compactification $\overline{\Tb^*}M$ of $\Tb^*M$, where the fibers are replaced by their radial compactification. The new boundary of $\overline{\Tb^*}M$ at fiber infinity is the \emph{b-cosphere bundle} $\Sb^*M$; it still possesses the compactification of the `old' boundary $\overline{\Tb^*}_XM$, see Figure~\ref{fig:Tb*M}. $\Sb^*M$ is naturally the quotient of $\Tb^*M\setminus o$ by the $\RR^+$-action of dilation in the fibers of the cotangent bundle. Many sets that we will consider below are conic subsets of $\Tb^*M\wozero$, and we will often view them as subsets of $\Sb^*M$.
\begin{figure}[!ht]
  \centering
  \includegraphics{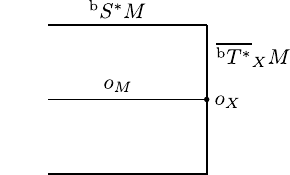}
  \caption{The radially compactified cotangent bundle $\overline{\Tb^*}M$ near $\overline{\Tb^*}_XM$; the cosphere bundle $\Sb^*M$, which is the boundary at fiber infinity of $\overline{\Tb^*}M$, is also shown, as well as the zero section $o_M\subset\overline{\Tb^*}M$ and the zero section over the boundary $o_X\subset\overline{\Tb^*}_XM$.}
  \label{fig:Tb*M}
\end{figure}

\item For $a\in\CI(\Tb^* M)$, we say $a\in S^m(\Tb^* M)$ if $a$ satisfies
  \[
    |\pa_z^\alpha\pa_\zeta^\beta a(z,\zeta)|\leq C_{\alpha\beta}\la\zeta\ra^{m-|\beta|}\tn{ for all multiindices }\alpha,\beta
  \]
  in any coordinate chart, where $z$ are coordinates in the base and $\zeta$ coordinates in the fiber; more precisely, in local coordinates $(\tau,x)$ near $X$, we take $\zeta=(\sigma,\xi)$, where we write b-covectors as
  \[
    \sigma\,\frac{d\tau}{\tau}+\sum_j\xi_j\,dx_j.
  \]
  We define the quantization $\Op(a)$ of $a$, acting on smooth functions $u$ supported in a coordinate chart, by
  \begin{align*}
    \Op(a)u(\tau,x)=(2\pi)^{-n}\int &e^{i(\tau-\tau')\tilde\sigma+i(x-x')\xi}\phi\left(\frac{\tau-\tau'}{\tau}\right) \\
	  &\hspace{1cm}\times a(\tau,x,\tau\tilde\sigma,\xi)u(\tau',x')\,d\tau'\,dx'\,d\tilde\sigma\,d\xi,
  \end{align*}
  where the $\tau'$-integral is over $[0,\infty)$, and
  $\phi\in\CIc((-1/2,1/2))$ is identically $1$ near $0$.\footnote{The
    cutoff $\phi$ ensures that these operators lie in the `small
    b-calculus' of Melrose, in particular that such quantizations act on weighted b-Sobolev spaces, defined below.} For general $u$, define $\Op(a)u$ using a partition of unity. We write $\Op(a)\in\Psib^m(M)$. We say that $a$ is a \emph{symbol} of $\Op(a)$. The equivalence class of $a$ in $S^m(\Tb^*M)/S^{m-1}(\Tb^*M)$ is invariantly defined on $\Tb^*M$ and is called the \emph{principal symbol} of $\Op(a)$. We will tacitly assume that all our operators have homogeneous principal symbols.
\item If $A\in\Psib^{m_1}(M)$ and $B\in\Psib^{m_2}(M)$, then $[A,B]\in\Psib^{m_1+m_2-1}(M)$, and its principal symbol is $\frac{1}{i}\sH_a b\equiv\frac{1}{i}\{a,b\}$, where the Hamilton vector field $\sH_a$ of the principal symbol $a$ of $A$ is the extension of the Hamilton vector field from $T^*M^\circ\wozero$ to $\Tb^*M\wozero$, which is a homogeneous degree $m-1$ vector field on $\Tb^*M\setminus o$ tangent to the boundary $\Tb^*_X M$. In local coordinates $(\tau,x,\sigma,\xi)$ on $\Tb^*M$ as above, this has the form
  \begin{equation}\label{eq:b-Ham-vf}
    \sH_a=(\pa_\sigma a)(\tau\pa_\tau)-(\tau\pa_\tau a)\pa_\sigma+\sum_j \big((\pa_{\xi_j}a)\pa_{x_j}-(\pa_{x_j}a)\pa_{\xi_j}\big).
  \end{equation}
\item We define bicharacteristics completely analogously to the semiclassical setting.
\item The \emph{microsupport} $\WFb'(A)\subset\Tb^*M\wozero$ of $A=\Op(a)\in\Psib^m(M)$ is the complement of the set of all $\rho\in\Tb^*M\wozero$ such that $a$ is rapidly decaying in a conic neighborhood around $\rho$. Note that $\WFb'(A)$ is conic, hence we will also view it as a subset of $\Sb^*M$.
\item Fix a \emph{b-density} on $M$, which is locally of the form $a\left|\frac{d\tau}{\tau}\,dz\right|$, $a>0$.
\item Define the \emph{b-Sobolev space} $\Hb^k(M)$ for $k\in\Z_{\geq 0}$ by
  \[
    \Hb^k(M)=\{u\in L^2(M)\colon X_1\cdots X_k u\in L^2(M), X_1,\ldots,X_j\in\Vb(M)\},
  \]
  and for general $k\in\R$ by duality and interpolation. Moreover, define the weighted b-Sobolev spaces $\Hb^{s,\alpha}(M):=\tau^\alpha\Hb^s(M)$ for $s,\alpha\in\R$, where $\tau$ is a boundary defining function, i.e.\ $\tau=0$ at $X$ and $d\tau\neq 0$ there. Every b-pseudodifferential operator $A\in\Psib^m(M)$ is a bounded operator $A\colon\Hb^{s,\alpha}(M)\to\Hb^{s-m,\alpha}(M)$, $s,\alpha\in\R$.
\item For $A\in\Psib^m(M)$ with principal symbol $a\in S^m(\Tb^*M)$, we say that $A$ is \emph{elliptic} at $\rho\in\Tb^*M\wozero$ if there is a constant $C>0$ such that $|a(z,\zeta)|\geq C|\zeta|^m$ for $(z,\zeta)$ in a conic neighborhood of $\rho$. The \emph{characteristic set} of $A$ is the complement (in $\Tb^*M\wozero$) of the set of all $\rho$ at which $A$ is elliptic.
\item For $u\in\Hb^{-\infty,\alpha}(M)$, define its \emph{$\Hb^{s,\alpha}$ wave front set} $\WFb^{s,\alpha}(u)\subset\Tb^*M\wozero$ as the complement of the set of all $\rho\in\Tb^*M\wozero$ for which there exists $a\in S^0(\Tb^*M)$ elliptic at $\rho$ such that $\Op(a)u\in\Hb^{s,\alpha}(M)$. In particular, $\WFb^{s,\alpha}(u)=\emptyset$ if and only if $u\in\Hb^{s,\alpha}(M)$.
\item \emph{Microlocal elliptic regularity} states that if $Au=f$ with $A\in\Psib^m(M)$, $u,f\in\Hb^{-\infty,\alpha}(M)$, $\rho\notin\WFb^{s-m,\alpha}(f)$ and $A$ is elliptic at $\rho$, then $\rho\notin\WFb^{s,\alpha}(u)$.
\item If $A\in\Psib^m(M)$ has a principal symbol with non-positive imaginary part, $u,f\in\Hb^{-\infty,\alpha}(M)$, $Au=f$, moreover $\rho\notin\WFb^{s,\alpha}(u)$ and $\gamma_\rho([0,T])\cap\WFb^{s-m+1,\alpha}(f)=\emptyset$ for some $T>0$, then the \emph{propagation of singularities} states that $\gamma_\rho(T)\notin\WFb^{s,\alpha}(u)$.
\end{itemize}

\subsection{Setup, statement and proof of the result}
\label{SubsecBResult}

Suppose $\cP\in\Psib^m(M)$, $\cP-\cP^*\in\Psib^{m-2}(M)$. Let $p$ be the principal symbol of $\cP$, which is thus a homogeneous degree $m$ function on $\Tb^*M\setminus o$, which we assume to be {\em real-valued}. Let $\tilde\rho$ denote a homogeneous degree $-1$ defining function of $\Sb^*M$. Then the rescaled Hamilton vector field
$$
V=\tilde\rho^{m-1}\sH_p
$$
is a $\CI$ vector field on $\overline{\Tb^*}M$ away from the 0-section, and it is tangent to all boundary faces. The characteristic set $\Sigma$ is the zero-set of the smooth function $\tilde\rho^m p$ in $\Sb^*M$. We will, somewhat imprecisely, refer to the flow of $V$ in $\Sigma\subset\Sb^*M$ as the Hamilton, or (null)bicharacteristic flow; its integral curves, the (null)bicharacteristics, are reparameterizations of those of the Hamilton vector field $\sH_p$, projected by the quotient map $\Tb^*M\setminus o\to\Sb^*M$.

We first work microlocally near the trapped set, namely assume that
\begin{enumerate}
\item $\Gamma\subset \Sigma\cap\Sb^*_XM$ is a smooth submanifold disjoint from the image of $T^*X\setminus o$ (so $\tau D_\tau$ is elliptic near
$\Gamma$),
\item $\Gamma_+$ is a smooth submanifold of $\Sigma\cap\Sb^*_X M$ in a neighborhood $U_1$ of $\Gamma$,
\item $\Gamma_-$ is a smooth submanifold of $\Sigma$ transversal to $\Sigma\cap\Sb^*_XM$ in $U_1$,
\item $\Gamma_+$ has codimension $2$ in $\Sigma$, $\Gamma_-$ has codimension $1$,
\item $\Gamma_+$ and $\Gamma_-$ intersect transversally in $\Sigma$ with $\Gamma_+\cap\Gamma_-=\Gamma$,
\item the rescaled Hamilton vector field $V=\tilde\rho^{m-1}\sH_p$ is tangent to both $\Gamma_+$ and $\Gamma_-$, and thus to $\Gamma$.
\end{enumerate}

\begin{figure}[!ht]
  \centering
  \includegraphics{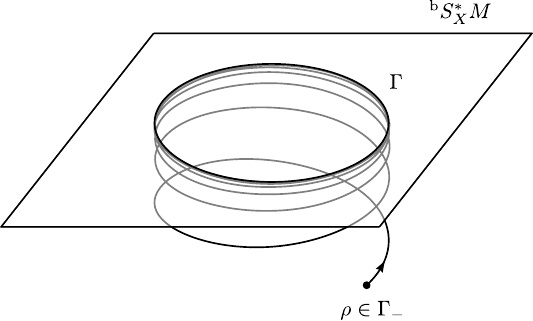}
  \caption{An exemplary situation with trapping: Shown are the (projection from $\Sb^*M$ to the base $M$ of the) trapped set $\Gamma$, the b-cosphere bundle over $X$ as well as a forward bicharacteristic starting at a point $\rho\in\Gamma_-$.}
  \label{FigBTrapping}
\end{figure}

We assume that $\Gamma_+$ is backward trapped for the
Hamilton flow (i.e.\ bicharacteristics in $\Gamma_+$ near $\Gamma$
tend to $\Gamma$ as the parameter goes to $-\infty$), i.e.\ is
the unstable manifold of $\Gamma$, while $\Gamma_-$ is forward
trapped, i.e.\ is the stable manifold of $\Gamma$, see Figure~\ref{FigBTrapping}; indeed, we assume
a quantitative version of this. (There is a completely analogous
statement if $\Gamma_+$ is forward trapped and $\Gamma_-$ is backward
trapped: replacing $\cP$ by $-\cP$ preserves all assumptions, but
reverses the Hamilton flow.)
To state this, let $\phi_-$ be a defining function of $\Gamma_-$, and let $\phi_+\in\CI(\Sb^*M)$ be a defining
function of $\Gamma_+$ in $\Sb^*_XM$; thus $\Gamma_+$ is defined
within $\Sb^*M$ by $\tau=0,\phi_+=0$.  Notice that $V$ being to tangent to $\Sb^*_X M$ (due to
\eqref{eq:b-Ham-vf}) implies that $V\tau$ is a multiple of $\tau$; we
assume that, near $\Gamma$,
\begin{equation}\label{eq:V-tau-neg}
V\tau=-c_\pa^2\tau,\ c_\pa>0;
\end{equation}
this is consistent with the stability of $\Gamma_-$. By the tangency requirement, with
$$
\hat p_0=\tilde\rho^m p,
$$
$V\phi_-=\alpha_-\phi_-+\nu_- \hat p_0$, $\alpha_-$ smooth; notice that changing
$\phi_-$ by a smooth non-zero multiple $f$ gives
$V(f\phi_-)=\alpha_-f\phi_-+\nu_-f\hat p_0 +(Vf)\phi_-$, so $\alpha_-$ depends on the
choice of $\phi_-$.
On the other hand, the tangency requirement gives
$V\phi_+=\alpha_+\phi_++\beta_+\tau+\nu_+\hat p_0$.
For the sake of conciseness, rather than stating the assumptions on the Hamilton flow as in
\cite{Wu11}, we assume directly that $\phi_\pm$ satisfy
\begin{equation}\label{eq:V-phi-pm}
V\phi_-=c_-^2\phi_-+\nu_-\hat p_0,\
V\phi_+=-c_+^2\phi_++\beta_+\tau+\nu_+\hat p_0,
\end{equation}
with $c_\pm>0$ smooth near $\Gamma$, $\beta_+,\nu_\pm$ smooth near
$\Gamma$ and
\begin{equation}\label{eq:phi-pm-Poisson}
\ \{\phi_+,\phi_-\}>0
\end{equation}
near $\Gamma$.
Let $U_0\subset\overline{U_0}\subset  
U_1$ be a neighborhood of  
$\Gamma$ such that the Poisson bracket in \eqref{eq:phi-pm-Poisson} as  
well as $c_\pm$ have positive lower bounds.

Now, given a boundary defining function $\tau$, $\Sb^*_XM\setminus S^*X$ (where $S^*X$ is the image of $T^*X$
under the $\RR^+$-dilation quotient) can be identified with
$(\frac{d\tau}{\tau}+T^*X)\cup(-\frac{d\tau}{\tau}+T^*X)$ (which in turn can be identified with two copies of $T^*X$)
since $\Tb^*_XM=\mathrm{Span}\big\{\frac{d\tau}{\tau}\big\}\oplus
T^*X$, and each $\RR^+$-orbit outside $T^*X$ intersects
$(\frac{d\tau}{\tau}+T^*X)\cup(-\frac{d\tau}{\tau}+T^*X)$ in a unique point. This provides the
connection between the b- and the semiclassical perspectives, i.e.\
analysis on $\Tb^*_X M$ and that of $T^*X$. In fact, if $p$ is a homogeneous degree
$m$ function, then $\tilde\rho^m p$, where $\tilde\rho$ can be
taken as the reciprocal of the absolute value of the symbol of $\tau D_\tau$ in this region
(which is well-defined, independent of choices),
gives a function on $\{\pm 1\}\times T^*X$; this is exactly the semiclassical
rescaling, with $\tilde\rho^m p$ the semiclassical principal symbol (depending on the parameter $\pm 1$) of
the rescaled operator family (cf.\ \cite[\S2.1]{Va12}).

Notice that if we merely assume the normal hyperbolicity within
$\Sb^*_XM$, in the sense of this identification with $T^*X$, as in \cite[\S1.2]{Wu11}, then \cite[Lemma~4.1]{Wu11}, as corrected in \cite{Wu11c},
actually gives such defining functions $\phi_\pm^0$ {\em within}
$\Sb_X^*M$ (i.e.\ letting $\tau=0$); taking an arbitrary extension in case of $\phi_+$, and an
extension which is a defining function in case of $\Gamma_-$, all the
requirements above are satisfied. In particular, Kerr and Kerr-de Sitter spaces
satisfy these assumptions, as do their perturbations when the angular
momentum $|a|$ is small; see \cite[Proposition~2.1]{Wu11} for the Kerr setting and
\cite[\S6]{Va12} for the Kerr-de Sitter one. Indeed, in the Kerr
case the full range of $|a|<M$, $M$ the black hole mass, satisfies the
hypotheses, as shown by Dyatlov \cite{Dyatlov:Asymptotics}.

There is an asymmetry between the roles of $\phi_\pm$
and $\tau$, and thus we consider the parabolic defining function
$$
\rho_+=\phi_+^2+\newconst\tau
$$
for $\Gamma_+$, $\newconst>0$, to be chosen. Then, near $\Gamma$,
\begin{equation}\begin{split}\label{eq:Gamma_+-quad-def}
\hat\rho_+=V\rho_+&=-2c_+^2\phi_+^2+2\beta_+\phi_+\tau+2\nu_+\phi_+\hat p_0-\newconst c_\pa^2\tau\\
&=-2c_+^2\phi_+^2-(\newconst c_\pa^2-2\beta_+\phi_+)\tau+2\nu_+\phi_+\hat p_0\\
&\leq -\tilde c_+^2\rho_++2\nu_+\phi_+\hat p_0,\qquad \tilde c_+>0,
\end{split}\end{equation}
if $\newconst>0$ is chosen sufficiently large, consistently with the
forward trapped nature of $\Gamma_-$. (Here the term with $\hat p_0$
is considered harmless as one essentially restricts to the
characteristic set, $\hat p_0=0$.)
Also, note that one can use\footnote{Indeed, in the semiclassical
  setting, after Mellin transforming this problem, $|\sigma|^{-1}$ plays the role
  of the semiclassical parameter $h$, which in that case {\em commutes}
  with the operator.} the reciprocal $\tilde\rho=|\sigma|^{-1}$ of
the principal symbol $\sigma$ of $\tau D_\tau$ as the local defining
function of $\Sb^*M$ as
fiber-infinity in $\Tb^*M$ near $\Gamma$; then
\begin{equation}\label{eq:V-tilde-rho}
V\tilde\rho=\tilde\alpha\tilde\rho\tau
\end{equation}
for some $\tilde\alpha$ smooth in view of \eqref{eq:b-Ham-vf}.

Similar to the normally isotropic spaces in the semiclassical setting, we introduce spaces which are {\em normally isotropic at $\Gamma$}.\footnote{Note that $\Tb^*M$ is {\em not} a symplectic manifold (in a natural way) since the symplectic form on $\Tb^*_{M^\circ}M$ does not extend smoothly to $\Tb^*M$. Thus, the word `normally isotropic' is not completely justified; we use it since it reflects that in the analogous semiclassical setting, see \cite{Wu11}, the set $\Gamma$ is symplectic, and the origin in the symplectic orthocomplement $(T_\alpha\Gamma)^\perp$ of $T_\alpha\Gamma$, which is also symplectic, is isotropic within $(T_\alpha\Gamma)^\perp$.} Concretely, let $Q_\pm\in\Psib^0(M)$ have principal symbol $\phi_\pm$ as before, $\hat P_0\in\Psib^0(M)$ have principal symbol $\hat p_0$ and let $Q_0\in\Psib^0(M)$ be elliptic, with real principal symbol for convenience, on $U_0^c$ (and thus nearby). Define the (global) b-normally isotropic spaces at $\Gamma$ of order $s$, $\bnormiso^s$, by the norm
\begin{equation}
\label{eq:b-norm-isotropic}
\|u\|_{\bnormiso^s}^2=\|Q_0 u\|^2_{\Hb^s}+\|Q_+u\|^2_{\Hb^s}+\|Q_-u\|^2_{\Hb^s}+\|\tau^{1/2} u\|^2_{\Hb^s}+\|\hat P_0 u\|^2_{\Hb^s}+\|u\|^2_{\Hb^{s-1/2}},
\end{equation}
and let $\bnormiso^{*,-s}$ be the dual space relative to $L^2$, which is thus\footnote{\label{FootDual}We refer to \cite[Appendix~A]{Melrose-Vasy-Wunsch:Corners} for a general discussion of the underlying functional analysis. In particular, Lemma~A.3 there essentially gives the density of $\dCI(M)$ in $\bnormiso^s(M)$: one can simply drop the subscript `e' in the statement of that lemma to conclude that $\Hb^\infty(M)$ (so in particular $\Hb^s(M)$) is dense in $\bnormiso^s(M)$, and then the density of $\dCI(M)$ in $\Hb^{s'}(M)$ for any $s'$ completes the argument. The completeness of $\bnormiso^s(M)$ follows from the continuity of $\Psib^0(M)$ on $\Hb^{s-1/2}(M)$.}
$$
Q_0\Hb^{-s}+Q_+\Hb^{-s}+Q_-\Hb^{-s}+\tau^{1/2}\Hb^{-s}+\hat P_0\Hb^{-s}+\Hb^{-s+1/2}.
$$
Note that microlocally away from $\Gamma$, $\bnormiso^s$ is just the standard $\Hb^s$ space while $\bnormiso^{*,-s}$ is $\Hb^{-s}$ since at least one of $Q_0$, $Q_\pm$ and $\tau$ is elliptic. Moreover, $\Psib^k(M)\ni A:\bnormiso^s\to\bnormiso^{s-k}$ is continuous since $[Q_+,A]\in\Psib^{k-1}(M)$ etc.; the analogous statement also holds for the dual spaces.  Further, the last term in \eqref{eq:b-norm-isotropic} can be replaced by $\|u\|^2_{\Hb^{s-1}}$ as $i[Q_+,Q_-]=B^*B+R$, $B\in\Psib^{-1/2}(M)$, $R\in\Psib^{-2}(M)$, using the same argument as in the semiclassical setting (however, it cannot be dropped altogether unlike in the semiclassical setting!).

\begin{rmk}
  The notation $\bnormiso^s(M)$ is justified for the space is independent of the particular defining functions $\phi_\pm$ chosen; near $\Gamma$ any other choice would replace $\phi_\pm$ by smooth non-degenerate linear combinations plus a multiple of $\tau$ and of $\hat p_0$, denote these by $\tilde\phi_\pm$, and thus the corresponding $\tilde Q_\pm$ can be expressed as
  \[
    B_+Q_++B_- Q_-+B_\pa\tau+\hat B\hat P_0+B_0Q_0+R,\quad
    B_\pm,B_0,B_\pa,\hat B\in\Psib^0(M),\ R\in\Psib^{-1}(M),
  \]
  so the new norm can be controlled by the old norm, and conversely in view of the non-degeneracy.
\end{rmk}

Our result is then:

\begin{thm}\label{thm:b-normiso-propagation}
  With $\cP,\bnormiso^s,\bnormiso^{*,s}$ as above, for any neighborhood $U$ of $\Gamma$ and for any $N$ there exist $B_0\in\Psib^0(M)$ elliptic at $\Gamma$ and $B_1,B_2\in\Psib^0(M)$ with $\WFb'(B_j)\subset U$, $j=0,1,2$, $\WFb'(B_2)\cap\Gamma_+=\emptyset$ and $C>0$ such that
  \begin{equation}\label{eq:b-normiso}
  \|B_0 u\|_{\bnormiso^s}\leq \|B_1\cP u\|_{\bnormiso^{*,s-m+1}}+\|B_2u\|_{\Hb^s}+C\|u\|_{\Hb^{-N}},
  \end{equation} 
  i.e.\ if all the functions on the right hand side are in the indicated spaces: $B_1\cP u\in \bnormiso^{*,s-m+1}$, etc., then $B_0 u\in\bnormiso^s$, and the inequality holds.
  
  The same conclusion also holds if we assume $\WFb'(B_2)\cap\Gamma_-=\emptyset$ instead of $\WFb'(B_2)\cap\Gamma_+=\emptyset$.
  
  Finally, if $r<0$, then, with $\WFb'(B_2)\cap\Gamma_+=\emptyset$, \eqref{eq:b-normiso} becomes
  \begin{equation}\label{eq:b-normiso-offspect}
    \|B_0 u\|_{\Hb^{s,r}}\leq \|B_1\cP u\|_{\Hb^{s-m+1,r}}+\|B_2u\|_{\Hb^{s,r}}+C\|u\|_{\Hb^{-N,r}},
  \end{equation}
  while if $r>0$, then, with $\WFb'(B_2)\cap\Gamma_-=\emptyset$,
  \begin{equation}\label{eq:b-normiso-offspect-reverse}
    \|B_0 u\|_{\Hb^{s,r}}\leq \|B_1\cP u\|_{\Hb^{s-m+1,r}}+\|B_2u\|_{\Hb^{s,r}}+C\|u\|_{\Hb^{-N,r}},
  \end{equation}
\end{thm}

\begin{rmk}
Note that the weighted versions
\eqref{eq:b-normiso-offspect}-\eqref{eq:b-normiso-offspect-reverse}
use {\em standard} weighted b-Sobolev spaces; this corresponds to
non-trapping semiclassical estimates if the subprincipal symbol has the
correct, definite, sign at $\Gamma$.
\end{rmk}

\begin{proof}
We may assume that $U\subset U_0$ is disjoint from a neighborhood of
$\WFb'(Q_0)$, and thus ignore $Q_0$ in the definition of $\bnormiso^s$ below.

We first prove that there exist $B_0,B_1,B_2$ as above and
$B_3\in\Psib^0(M)$ with $\WFb'(B_3)\subset U$ such that
\begin{equation}\label{eq:weaker-est-b-trap}
\|B_0 u\|_{\bnormiso^s}\leq \|B_1\cP u\|_{\bnormiso^{*,s-m+1}}+\|B_2u\|_{\Hb^s}+\|B_3u\|_{\Hb^{s-1}}+C\|u\|_{\Hb^{-N}}.
\end{equation}
An iterative argument will then prove the theorem.

The proof is a straightforward modification of the construction in the
semiclassical setting above, replacing $\phi_+^2$ by
$\phi_+^2+\newconst\tau$, $\newconst>0$ large, in accordance with
\eqref{eq:Gamma_+-quad-def}.

We start by pointing out that for any $\tilde B_0\in\Psib^0(M)$ and any
$\tilde B_3\in\Psib^0(M)$ elliptic on $\WFb'(\tilde B_0)$, we have
\begin{equation}
\label{EqEllEst}
  \|\hat P_0 \tilde B_0 u\|_{\Hb^s}\leq C\|\tilde B_0\cP u\|_{\Hb^{s-m}}+C'\|\tilde B_3u\|_{\Hb^{s-1}},
\end{equation}
by simply using that $\hat P_0$ is an elliptic
multiple of $\cP$ modulo $\Psib^{-1}(M)$. Since
$\|\tilde B_0\cP u\|_{\Hb^{s-m}}\leq C\|\tilde B_0 \cP u\|_{\bnormiso^{*,s-m}}$, the
$\hat P_0$ contribution to $\|\tilde B_0 u\|_{\bnormiso^s}$ in
\eqref{eq:weaker-est-b-trap} is thus automatically controlled.

So let $\chi_0(t)=e^{-\digamma/t}$ for $t>0$,
$\chi_0(t)=0$ for $t\leq 0$, with $\digamma>0$ (large) to be specified, $\chi\in\CI_c([0,\infty))$ be identically $1$ near $0$ with $\chi'\leq
0$, and indeed with $\chi'\chi=-\chi_1^2$, $\chi_1\geq 0$,
$\chi_1\in\CI_c([0,\infty))$,
and let $\psi\in\CI_c(\RR)$ be identically $1$ near $0$.
As we use the Weyl quantization,\footnote{Again, the Weyl quantization
is irrelevant: If $A\in\Psib^m(X)$ and the
  principal symbol of $A$ is real then the real part
  of the subprincipal symbol is defined independently of choices,
  which suffices below.} we write $\cP$ as the Weyl
quantization of $p=p_0+\tilde\rho p_1$, with $\tilde\rho p_1$ of order $m-1$.
Let
\begin{equation}\label{eq:a-def-normiso}
a=\tilde\rho^{-s+(m-1)/2}\chi_0(\rho_+-\phi_-^2+\kappa)\chi(\rho_+)\psi(\tilde\rho^m p),
\end{equation}
$\kappa>0$ small.
Notice that on $\supp a$, if $\chi$ is supported in $[0,R]$,
$$
\rho_+\leq R,\ \phi_-^2\leq\rho_++\kappa=R+\kappa,
$$
so $a$ is localized near $\Gamma$ if $R$ and $\kappa$ are taken
sufficiently small.
In particular, the argument of $\chi_0$ is bounded above by
$R+\kappa$, so given any $\newconst_0>0$ one can take $\digamma>0$
large so that
$$
\chi_0'\chi_0-\newconst_0\chi_0^2=b^2\chi_0'\chi_0,
$$
with $b\geq 1/2$, $\CI$, on the range of the argument of $\chi_0$.

In fact, we also need to regularize, namely introduce
\begin{equation}\label{eq:trap-commutant-regularize}
a_\ep=(1+\ep\tilde\rho^{-1})^{-2} a,\ \ep\in[0,1],
\end{equation}
which is a symbol of order $s-(m-1)/2-2$ for $\ep>0$, and is uniformly
bounded in symbols of order $s-(m-1)/2$ as $\ep$ varies in $[0,1]$. In
order to avoid more cumbersome notation below, we ignore the
regularizer and work directly with $a$; since the regularizer
gives the same kind of contributions to the commutator as the weight
$\tilde\rho^{-s+(m-1)/2}$, these contributions can be dominated in
exactly the same way.

Then, with $p=p_0+\tilde\rho p_1$ as above, $W=\tilde\rho^{m-2}\sH_{\tilde\rho
  p_1}$, which is a smooth vector field near $\Sb^*M$ as $\tilde\rho
p_1$ is order $m-1$, noting $W\tilde\rho=\tilde\alpha_1\tau\tilde\rho$
similarly to \eqref{eq:V-tilde-rho}, and $W\tau=\alpha_{\pa,1}\tau$ by
the tangency of $W$ to $\tau=0$,
\begin{equation}\begin{split}\label{eq:Ham-deriv-b-normiso}
\frac{1}{4}\sH_p(a^2)=&-
(-\hat\rho_+/2+c_-^2\phi_-^2+\nu_-\phi_-\hat p_0-\tilde\rho\phi_+(W\phi_+)-\tilde\rho\newconst\alpha_{\pa,1}\tau+\tilde\rho\phi_-(W\phi_-))\\
&\qquad\qquad\qquad\times \tilde\rho^{-2s} (\chi_0\chi_0')(\rho_+-\phi_-^2+\kappa)
\chi(\rho_+)^2\psi(\tilde\rho^m p)^2\\
&+\frac{1}{4}(-2s+m-1)\tilde\rho^{-2s}(\tilde\alpha+\tilde\rho\tilde\alpha_1)\tau
\chi_0(\rho_+-\phi_-^2+\kappa)^2 \chi(\rho_+)^2\psi(\tilde\rho^m p)^2\\
&+\frac{1}{2}\tilde\rho^{-2s}(\hat\rho_++\tilde\rho
W\rho_+)(\chi'\chi)(\rho_+)\chi_0(\rho_+-\phi_-^2+\kappa)^2\psi(\tilde\rho^m
p)^2\\
&+\frac{m}{2}(\tilde\alpha+\tilde\rho\tilde\alpha_1)\tilde\rho^{-2s} (\tilde\rho^m p)\tau
\chi_0(\rho_+-\phi_-^2+\kappa)^2 \chi(\rho_+)^2(\psi\psi')(\tilde\rho^m p).
\end{split}\end{equation}
A key point is that the second term on the right hand side, given by the weight
$\tilde\rho^{-2s+m-1}$ being differentiated, can be absorbed into the first by making
$\digamma>0$ large so that $\hat\rho_+\chi_0'(\rho_+-\phi_-^2+\kappa)$
dominates
$$
|-2s+m-1||\tilde\alpha|\tau \chi_0(\rho_+-\phi_-^2+\kappa)
$$
on $\supp a$, which can be arranged as $|-2s+m-1||\tilde\alpha|\tau$
is bounded by a sufficiently large multiple of $\hat\rho_+$ there.
Thus,
\begin{equation}\label{eq:Ham-deriv-b-normiso-squares}
\frac{1}{4}\sH_p(a^2)=-c_+^2a_+^2-c_-^2a_-^2-a_\pa^2+2g_+a_++2g_-a_-+e+\tilde
e+2a_+j_+ p+2a_-j_-p
\end{equation}
with
\begin{align*}
&a_\pm=\tilde\rho^{-s}\phi_\pm \sqrt{(\chi_0\chi_0')
  (\rho_+-\phi_-^2+\kappa)}\chi(\rho_+)\psi(\tilde\rho^m p),\\
&a_\pa=\tilde\rho^{-s}\tau^{1/2}\Big((\newconst
(c_\pa^2/2)-\beta_+\phi_+ -\tilde\rho \newconst a_{\pa,1})(\chi_0\chi_0')
  (\rho_+-\phi_-^2+\kappa)\\
&\qquad\qquad\qquad -\frac{1}{4}(-2s+m-1) (\tilde\alpha+\tilde\rho\tilde\alpha_1) \chi_0(\rho_+-\phi_-^2+\kappa)^2\Big)^{1/2}\chi(\rho_+)\psi(\tilde\rho^m p),\\
&g_\pm=\pm\frac{1}{2}\tilde\rho^{-s+1}((W\phi_\pm)-\nu_\pm\tilde\rho^{m-1}p_1) \sqrt{(\chi_0\chi_0')
  (\rho_+-\phi_-^2+\kappa)}\chi(\rho_+)\psi(\tilde\rho^m p),\\
&e=-\frac{1}{2}\tilde\rho^{-2s} (\hat\rho_++\hat\rho W\rho_+)\chi_1(\rho_+)^2
\chi_0(\rho_+-\phi_-^2+\kappa)^2\psi(\tilde\rho^m p)^2,\\
&\tilde e=\frac{m}{2}\tilde\rho^{-2s}(\tilde\rho^m p) (\tilde\alpha+\tilde\rho\tilde\alpha_1)\tau
\chi_0(\rho_+-\phi_-^2+\kappa)^2
\chi(\rho_+)^2(\psi\psi')(\tilde\rho^m p),\\
&j_\pm=\pm\frac{1}{2}\nu_\pm \tilde\rho^{-s+m}\sqrt{(\chi_0\chi_0') (\rho_+-\phi_-^2+\kappa)}\chi(\rho_+)\psi(\tilde\rho^m p);
\end{align*}
the square root in $a_\pa$ is
that of a non-negative quantity and is $\CI$ for $\newconst$ large (so that $\beta_+\phi_+$ can be absorbed into $\newconst (c_\pa^2/2)$) and $\digamma$ large (so that a small multiple of $\chi_0'$ can be used to dominate $\chi_0$), as discussed earlier. Moreover,
\begin{equation*}\begin{split}
&\supp e\subset\supp a,\ \supp e\cap\Gamma_+=\emptyset,\\
&\supp \tilde e\subset\supp a,\ \supp \tilde e\cap\Sigma=\emptyset.
\end{split}\end{equation*}
This gives, with the various operators being Weyl quantizations of the
corresponding lower case symbols,
\begin{equation}\begin{split}\label{eq:b-normiso-comm-def}
\frac{i}{4}[\cP,A^*A]=&-(C_+A_+)^*(C_+A_+)-(C_-A_-)^*(C_-A_-)-A_\pa^*A_\pa\\
&\qquad+G_+^*A_++A_+^*G_++G_-^*A_-+A_-^*G_-\\
&\qquad+E+\tilde E+A_+^*J_+\cP+\cP^*J_+^*A_++A_-^*J_-\cP+\cP^*J_-^*A_-+F
\end{split}\end{equation}
where now $A\in\Psib^{s-(m-1)/2}(M)$, $A_\pm,A_\pa\in\Psib^s(M)$, $G_\pm\in\Psib^{s-1}(M)$,
$E\in\Psib^{2s}(M)$, $\tilde E\in\Psib^{2s}(M)$, $J_\pm\in\Psib^{s-m}(M)$, $F\in\Psib^{2s-2}(M)$ with $\WFb'(F)\subset\supp
a$.

After this point the calculations repeat the semiclassical
argument: First using $\cP-\cP^*\in\Psib^{m-2}(M)$,
\begin{equation}\begin{split}\label{eq:pos-comm-est-trap-b}
&\|C_+A_+ u\|^2+\|C_-A_- u\|^2+\|A_\pa u\|^2\\
&\qquad\leq |\la Eu,u\ra|+|\la\tilde E u,u\ra|+|\la A\cP u,Au\ra|+2\|A_+ u\|\|G_+u\|+2\|A_-u\|\|G_- u\|\\
&\qquad\qquad\qquad+2|\la J_+\cP u,A_+u\ra|+2|\la J_-\cP u,A_-u\ra|+C_1\|\tilde F_1 u\|^2_{\Hb^{s-1}}+C_1\|u\|_{\Hb^{-N}}^2,
\end{split}\end{equation}
where we took $\tilde F_1\in\Psib^0(M)$ elliptic on $\WFb'(F)$ and with $\WFb'(\tilde
F_1)$ near $\Gamma$. Noting that $\WFb'(\tilde
E)\cap\Sigma=\emptyset$, the elliptic estimates give
$$
|\la\tilde
Eu,u\ra|\leq C\|B_1 \cP u\|_{\Hb^{s-m}}^2+C\|u\|_{\Hb^{-N}}^2
$$
if $B_1\in\Psib^0(M)$ is elliptic on $\supp \tilde e$.
Let $\Lambda\in\Psib^{(m-1)/2}(M)$ be elliptic with real principal
symbol $\lambda$, and let
$\Lambda^-\in\Psib^{-(m-1)/2}(M)$ be a parametrix for it so that
$\Lambda\Lambda^--\Id=R_0\in\Psib^{-\infty}(M)$. Then
\begin{equation*}\begin{split}
|\la A\cP u,Au\ra|&\leq|\la\Lambda^-A\cP
u,\Lambda^*Au\ra\|+|\la R_0A\cP u,Au\ra|\\
&\leq\frac{1}{2\ep}\|\Lambda^-A\cP
u\|^2_{\bnormiso^{*,0}}+\frac{\ep}{2}\|\Lambda^*Au\|^2_{\bnormiso^0}+C'\|u\|^2_{\Hb^{-N}}
\end{split}\end{equation*}
As $\Lambda^*A\in\Psib^s(M)$, for sufficiently small $\ep>0$,
$\frac{\ep}{2}\|\Lambda^*Au\|^2_{\bnormiso^0}$ can be absorbed
into\footnote{The point being that $A_+^*C_+^*C_+A_+-\ep
  A^*\Lambda Q_+^*Q_+\Lambda^*A$ has principal symbol $c_+^2a_+^2-\ep
  a^2\phi_+^2\lambda^2$ which can be written as the square of a real
  symbol for $\ep>0$ small in view of the main difference in vanishing
  factors in the two terms being that $\chi_0'$ in $a_+^2$ is replaced by $\chi_0$ in $a$,
  and thus the corresponding operator can be expressed as $\tilde C^*\tilde C$ for
  suitable $\tilde C$, modulo an element of $\Psib^{2s-2}(M)$, with
  the latter contributing to the $\Hb^{s-1}$ error term on the right
  hand side of \eqref{eq:weaker-est-b-trap}.}
$\|C_+A_+ u\|^2+\|C_-A_- u\|^2+\|A_\pa u\|^2$ plus $\|\tilde B_0\hat
P_0 u\|^2_{\Hb^s}$, and as discussed above, the latter already has the control
required for \eqref{eq:weaker-est-b-trap}. On the other hand,
taking $B_1\in\Psib^0(M)$ elliptic on $\WFb'(A)$, as
$\Lambda^-A\in\Psib^{s-m+1}(M)$,
$$
\|\Lambda^-A\cP u\|^2_{\bnormiso^{*,0}}\leq C''\|B_1\cP u\|_{\bnormiso^{*,s-m+1}}^2+C''\|u\|_{\Hb^{-N}}^2.
$$
Similarly, to deal with the $J_\pm$ terms on the right hand side of
\eqref{eq:pos-comm-est-trap-b}, one writes
\begin{equation*}\begin{split}
|\la J_\pm\cP u,A_\pm u\ra|&\leq \frac{1}{2\ep}\Big(\|B_1\cP
u\|^2_{\Hb^{s-m}}+C''\|u\|_{\Hb^{-N}}^2\Big)+\frac{\ep}{2}\|A_\pm u\|^2_{L^2}\\
&\leq \frac{1}{2\ep}\Big(\|B_1\cP
u\|^2_{\bnormiso^{*,s-m}}+C''\|u\|_{\Hb^{-N}}^2\Big)+\frac{\ep}{2}\|A_\pm u\|^2_{L^2},
\end{split}\end{equation*}
while the $G_\pm$ terms can be estimated by
$$
\ep\|A_+ u\|^2+\ep^{-1}\|G_+u\|^2+\ep\|A_- u\|^2+\ep^{-1}\|G_-u\|^2,
$$
and for $\ep>0$ sufficiently small, the $\|A_\pm u\|^2$ terms in both cases can
be absorbed into the left hand side of
\eqref{eq:pos-comm-est-trap-b} while the $G_\pm$ into the error
term. This gives, with $\tilde F_2$ having properties as $\tilde F_1$,
\begin{equation*}\begin{split}
&\|C_+A_+ u\|^2+\|C_-A_- u\|^2+\|A_\pa u\|^2\\
&\qquad\leq |\la Eu,u\ra|+C\|B_1\cP u\|_{\bnormiso^{*,s-m+1}}^2+C_2\|\tilde F_2 u\|_{\Hb^{s-1}}^2 +C_2\|u\|_{\Hb^{-N}}^2.
\end{split}\end{equation*}
By the remark before the statement of the theorem, if
$B_0\in\Psib^0(M)$ is such that
$\chi_0(\rho_+-\phi_-^2+\kappa)\chi(\rho_+)\psi(p)>0$ on $\WFb'(B_0)$,
$\|B_0 u\|^2_{\Hb^{s-1/2}}$ can be added to the left hand side at the cost
of changing the constant in front of $\|\tilde F_2 u\|_{\Hb^{s-1}}^2
+\|u\|_{\Hb^{-N}}^2$
on the
right hand side. Taking such $B_0\in\Psib^0(M)$,
and $B_1$ elliptic on $\WFb'(A)$ as before, $B_2\in\Psib^0(M)$ elliptic on $\WFb'(E)$ but
with $\WFb'(B_2)$ disjoint from $\Gamma_+$, we conclude that
$$
\|B_0 u\|_{\bnormiso^s}^2\leq C\|B_1\cP u\|_{\bnormiso^{*,s-m+1}}^2+C\|B_2 u\|^2_{\Hb^s}+C\|\tilde F_2 u\|^2_{\Hb^{s-1}}+C\|u\|_{\Hb^{-N}}^2,
$$
proving \eqref{eq:weaker-est-b-trap}, up to redefining $B_j$ by
multiplication by a positive constant. Recall that unless one makes
sufficient a priori assumptions on the
regularity of $u$, one actually needs to regularize, but as mentioned
after \eqref{eq:trap-commutant-regularize}, the regularizer is handled
in exactly the same manner as the weight.

Now in general, with $\chi$ as before, but supported in $[0,1]$ instead of $[0,R]$, let $\chi_R=\chi(\cdot/R)$ and write $a=a_{R,\kappa}$ to emphasize its dependence on these quantities. When $R$ and $\kappa$ are decreased, $\supp a_{R,\kappa}$ also decreases in $\Sigma$ in the strong sense that $0<R<R'$ and $0<\kappa<\kappa'$ imply that $a_{R',\kappa'}$ is elliptic on $\supp a_{R,\kappa}$ within $\Sigma$, and indeed globally if the cutoff $\psi$ is suitably adjusted as well. Thus, if $u\in\Hb^{-N}$, say, one uses first \eqref{eq:weaker-est-b-trap} with $s=-N+1$, and with $B_j$ given by the proof above, so the $B_3 u$ term is a priori bounded, to conclude that $B_0 u\in\bnormiso^s$ and the estimate holds, so in particular, $u$ is in $\Hb^{-N+1/2}$ microlocally near $\Gamma$ (concretely, on the elliptic set of $B_0$). Now one decreases $\kappa$ and $R$ by an arbitrarily small amount and applies \eqref{eq:weaker-est-b-trap} with $s=-N+3/2$; the $B_3 u$ term is now a priori bounded by the microlocal membership of $u$ in $\Hb^{-N+1/2}$, and one concludes that $B_0 u\in\bnormiso^{-N+3/2}$, so in particular $u$ is microlocally in $\Hb^{-N+1}$. Proceeding inductively, one deduces the first statement of the theorem, \eqref{eq:b-normiso}.

If one reverses the role of $\Gamma_+$ and $\Gamma_-$ in the statement
of the theorem, one simply reverses the roles of
$\rho_+=\phi_+^2+\newconst\tau$ and $\phi_-^2$ in the definition of
$a$ in \eqref{eq:a-def-normiso}. This reverses the signs of all terms on the right hand side of \eqref{eq:Ham-deriv-b-normiso}
 whose sign mattered below, and thus the signs of the first three
 terms on the right hand side of \eqref{eq:b-normiso-comm-def}, which
 then does not affect the rest of the argument.

In order to prove \eqref{eq:b-normiso-offspect}, one simply adds a
factor $\tau^{-2r}$ to the definition of $a$ in
\eqref{eq:a-def-normiso}. This adds a factor $\tau^{-2r}$ to every
term on the right hand side of \eqref{eq:b-normiso-comm-def}, as well as an additional term
$$
\frac{r}{2}\tau^{-2r}\tilde\rho^{-2s}c_\pa^2 \chi_0(\rho_+-\phi_-^2+\kappa)^2 \chi(\rho_+)^2\psi(p)^2,
$$
which for $r<0$ has the same sign as the terms whose sign was used
above, and indeed can be written as the negative of a square. Thus
\eqref{eq:Ham-deriv-b-normiso-squares} becomes
\begin{equation}\begin{split}\label{eq:Ham-deriv-b-normiso-squares-2}
\frac{1}{4}\sH_p(a^2)=&-c_+^2a_+^2-c_-^2a_-^2-a_\pa^2-a_r^2\\
&\qquad+2g_+a_++2g_-a_-+e+\tilde e+2j_+a_+p+2j_-a_-p
\end{split}\end{equation}
with
$$
a_r=\sqrt{\frac{-r}{2}}\tau^{-r}\tilde\rho^{-s}c_\pa\chi_0(\rho_+-\phi_-^2+\kappa) \chi(\rho_+)\psi(p),
$$
and all other terms as above apart from the additional factor of
$\tau^{-r}$ in the definition of $a_\pm$, etc. Since $a_r$ is actually
elliptic at $\Gamma$ when $r\neq 0$, this proves the desired estimate
(and one does not need to use the improved properties given by the
Weyl calculus!).

When the role of $\Gamma_+$ and $\Gamma_-$ is reversed, there is an
overall sign change, and thus $r>0$ gives the advantageous sign; the
rest of the argument is unchanged.
\end{proof}

\begin{rmk}
  As in the semiclassical setting, see Remark~\ref{RmkEll}, the estimate \eqref{eq:b-normiso} can be strengthened by adding the term $\|B_0\hat P_0 u\|_{\Hb^{s+1}}$ to the left hand side, which is controlled by elliptic regularity, likewise for \eqref{eq:b-normiso-offspect}--\eqref{eq:b-normiso-offspect-reverse}. A more natural way of phrasing such an improvement is to use `coisotropic, normally isotropic' spaces $\tbnormiso^s$ and $\tbnormiso^{*,s}$ in the estimate \eqref{eq:b-normiso}, where the squared norm on $\tbnormiso^s$ is defined by
  \[
    \|u\|_{\tbnormiso^s}^2=\|Q_0 u\|^2_{\Hb^s}+\|Q_+u\|^2_{\Hb^s}+\|Q_-u\|^2_{\Hb^s}+\|\tau^{1/2} u\|^2_{\Hb^s}+\|\hat P_0 u\|^2_{\Hb^{s+1/2}}+\|u\|^2_{\Hb^{s-1/2}},
  \]
  i.e.\ strengthening the norm of $\hat P_0 u$ by a half, which strengthens the space and weakens its dual. To obtain the necessary elliptic estimate \eqref{EqEllEst} with the strengthened norms on the terms involving $\tilde B_0$, but keeping the norm on $\tilde B_3 u$ (which is required for the iterative argument at the end of the proof), one can choose $\tilde B_0$ with $\WFb'(I-\tilde B_0)\cap\Gamma=\emptyset$ so that $\tilde B_3$ can be chosen to be microsupported away from $\Gamma$, and thus $\|\tilde B_3 u\|_{\Hb^{s-1/2}}\leq C\|\tilde B_3 u\|_{\bnormiso^{s-1/2}}$ is controlled using \eqref{eq:b-normiso}, with the norm on $B_1 \cP u$ being $\|B_1\cP u\|_{\bnormiso^{*,s-m+1/2}}\leq C\|B_1\cP u\|_{\tbnormiso^{*,s-m+1}}$, and the error term being measured in $\Hb^{s-3/2}\supset\Hb^{s-1}$, as required.
\end{rmk}

\def\cprime{$'$}

\end{document}